\documentclass{article}

\usepackage{amsmath}
\usepackage{amsthm}
\usepackage{amsfonts}
\usepackage{amsbsy}
\usepackage{amssymb}
\newtheorem{theorem}{Theorem}
\newtheorem{proposition}{Proposition}
\newtheorem{corollary}{Corollary}

\newcommand{\R}{{\mathbb R}}

\newcommand{\set}[2]{ \left\{ #1 \ \left| \ #2 \right. \right\} }

\title{$L^p$-improving estimates for averages on polynomial curves}
\author{Philip T. Gressman\footnote{Partially supported by NSF grant DMS-0850791.}}
\begin{document}

\maketitle
\begin{abstract}
In the combinatorial method proving of $L^p$-improving estimates for averages along curves pioneered by Christ \cite{christ1998}, it is desirable to estimate the average modulus (with respect to some uniform measure on a set) of a polynomial-like function from below using only the value of the function or its derivatives at some prescribed point.   In this paper, it is shown that there is always a relatively large set of points (independent of the particular function to be integrated) for which such estimates are possible.
%
%
%This paper has two main parts.  In the first, an inequality is established which shows that, under fairly general circumstances, the average modulus (with respect to any regular probability measure) of a polynomial-like function on the real line can be estimated from below in terms of the value of the function or its derivatives at any single point amongst a relatively large set of .  The difficulty arises in showing that the set of points does not depend in any serious way on the particular function being integrated.  
Inequalities of this type are then applied to extend the results of Tao and Wright \cite{tw2003} to obtain endpoint restricted weak-type estimates for averages over curves given by polynomials.
\end{abstract}

The purpose of this paper is twofold.  First, a somewhat surprising inequality for $L^1(\mu)$-norms of polynomial-type functions on the real line will be established.   The most important special case of this inequality is as follows:
%The particular theorems of the first section deal with uniform estimation of integrals from below in a fairly general setting (in an effort to establish the weakest possible conditions necessary for the desired results to hold).  While these results appear to be interesting in their own right, it is the following special case that is likely to be of the broadest appeal:
\begin{theorem}
Suppose that $K \subset \R$ is measurable.  For any positive integer $n$ and any $0 < \epsilon < 1$, there is an interval $I$ with $|K \cap I| \geq \frac{1-\epsilon}{n} |K|$ (here $| \cdot |$ acting on sets denotes Lebesgue measure) and a constant $c_{n,\epsilon}$ such that \label{theorem0}
\begin{equation} \int_{K} |p(t)| dt \geq c_{n,\epsilon} |K|^{j+1} \sup_{t \in I} |p^{(j)}(t)| \label{ineq0} \end{equation}
for any polynomial $p$ of degree $n$ or less and any $j=0,\ldots,n$.
\end{theorem}
It is fairly trivial to show that such an interval $I$ can always be found when $p$ is given, but somewhat unexpected that $I$ exists independently of $p$.  It is also farily straightforward to construct a family of counterexamples to this inequality if one takes $\epsilon=0$, for example (which suggests that there must be a certain amount of subtlety involved in proving the positive result).  This theorem and its generalization to regular probability measures $\mu$ will be taken up in the first section.
%namely, that the value of this norm may be estimated from below by any individual pointwise value of the function or its derivatives at any point in some set which is largely independent of the particular function being integrated.  This set has the property that it has nontrivial $\mu$-measure.  

The second major purpose of this paper is to apply the inequality \eqref{ineq0} to answer a question of Tao and Wright \cite{tw2003} concerning $L^p$-improving bounds for averages along curves.  In that work, inequalities of the form \eqref{ineq0} appear naturally while carrying out the combinatorial methods of Christ \cite{christ1998}.  Tao and Wright were able to establish \eqref{ineq0} in the case when $K$ has some additional structure ($K$ was taken to be a central set of width $w$, see lemma 7.3); the price of requiring such structure was certain losses in exponents which made it impossible to obtain $L^p$-improving inequalities on the boundary of the type set.
Interpreted in the framework of that paper, theorem \ref{theorem0} indicates that the set $K$ does not need any significant structure (namely, $K$ does not need to be a central set of width $w$) for the desirable inequality \eqref{ineq0} to hold. As a consequence, theorem \ref{theorem0} alone makes it possible to prove the full range of restricted weak-type estimates for one-dimensional averaging operators given by polynomial curves (that is, when both the averaging operator and the dual operator can be described by polynomial functions in appropriate coordinate systems).

  This result is formulated in the standard bilinear way as follows: suppose $U$ is an open ball in $\R^{d+1}$ and one is given projections $\pi_1 : U \rightarrow \R^d$ and $\pi_2 : U \rightarrow \R^d$ such that the differentials $d \pi_1$ and $d \pi_2$ are surjective at every point.  The Radon-like operator $R$ associated to these projections is defined by duality as
\begin{equation}
 \int_{\R^{d}} R f(y) g(y) dy := \int_U f (\pi_1(x)) g(\pi_2(x)) \psi(x) dx \label{theop}
\end{equation}
where $\psi$ is some bounded cutoff function (not necessarily smooth) supported in $U$.
Next, let $X_1$ and $X_2$ be vector fields on $U$ which are nonvanishing and satisfy $d \pi_1 (X_1) = 0 = d \pi_2 (X_2)$,  and suppose that for all words $w = (w_1,w_2,\ldots,w_k)$ (each $w_j$ equals $1$ or $2$) of sufficient length, the commutator
\[ X_w := [ X_{w_1}, [ X_{w_2}, [ \cdots [X_{w_{k-1}},X_{w_k}] \cdots]]] \]
vanishes identically on $U$.  This geometric condition guarantees that both $R$ and $R^*$ are given by averages over polynomial curves.  Under these conditions the following theorem holds:
\begin{theorem}
Suppose that the vector fields $X_1$, $X_2$ are as assumed above.  Let $x_0 \in U$, and consider the mapping \label{averageop}
\[ \Phi_{x_0}(t_1,\ldots,t_{d+1}) := \exp(t_1 X_1) \circ \cdots \circ \exp(t_{d+1} X_{d+1})(x_0) \]
(where the periodicity convention $X_{j+2} = X_j$ is used).  Let $J_{x_0}(t)$ be the Jacobian determinant of this mapping (as a function of the parameters $t$).  If $\partial_t^\alpha J_{x_0}(t) \neq 0$ at $t=0$
%\[  \left. \left( \frac{\partial}{\partial t} \right)^\alpha \right|_{t=0} J_{x_0}(t) \neq 0 \]
for some multiindex $\alpha = (\alpha_1,\ldots,\alpha_{d+1})$, then the averaging operator $R$ given by \eqref{theop}
satisfies a restricted weak-type estimate 
\[ \left| \int_U \chi_F(\pi_1(x)) \chi_G(\pi_2(x)) \psi(x) dx \right| \leq C |F|^\frac{1}{p_1} |G|^\frac{1}{p_2} \]
when the support of $\psi$ is a sufficiently small neighborhood of $x_0$ and $\frac{1}{p_1} := \frac{A_1}{A_1 + A_2 -1}, \frac{1}{p_2} := \frac{A_2}{A_1+A_2 - 1}$ where
\begin{align*} A_1 & := \left\lceil \frac{d+1}{2} \right\rceil + \alpha_1 + \cdots + \alpha_{2\lceil \frac{d+1}{2} \rceil-1}   , \ 
A_2  := \left\lfloor \frac{d+1}{2} \right\rfloor + \alpha_2  + \cdots + \alpha_{2\lfloor \frac{d+1}{2} \rfloor}.
\end{align*}
\end{theorem}
It should be noted that the hypotheses of theorem \ref{averageop} (regarding the derivatives of the Jacobian determinant) are, in fact, equivalent to the H\"{o}rmander-type hypotheses used by Tao and Wright (see lemma 9.1 in \cite{tw2003} or sections 9 and 10 of the paper of Christ, Nagel, Stein, and Wainger \cite{cnsw1999} in which the double fibration curvature formulation $({\cal C}_{\Lambda})$ is shown to be equivalent to the Jacobian determinant formulation $({\cal C}_J)$).  In light of this equivalence, theorem \ref{averageop} successfully establishes restricted weak-type estimates on the boundary of the type set of the operator \eqref{theop} which were just missed in \cite{tw2003}.

 Regarding integral estimates and the related issue of sublevel sets, earlier results of particular interest to the problem at hand include the work of Carbery, Christ, and Wright \cite{ccw1999}, Phong, Stein, and Sturm \cite{pss2001}, and Phong and Sturm \cite{ps2000}, as well as many others.  In this paper, the attention will be exclusively focused on one-dimensional estimates, using methods similar to those employed by Carbery, Christ and Wright \cite{ccw1999} who built upon ideas of Arhipov, Karacuba and \v{C}ubarikov \cite{akc1979}.  Similar one-dimensional methods have also been employed by Rogers \cite{rogers2005} to obtain sharp constants for sublevel set estimates and van der Corput's lemma.

In the case of Radon-like transforms and averaging operators, the reader is referred to the papers of Tao and Wright \cite{tw2003} and Christ, Nagel, Stein, and Wainger \cite{cnsw1999} for more complete lists of references.  In this paper, the argument to be followed was originally devised by Christ \cite{christ1998}.  Tao and Wright \cite{tw2003} made important additions to the Christ argument which will, of course, be necessary to use here as well.  More recently, these ideas have been adapted to a variety of other contexts by Christ and Erdo{\v{g}}an \cite{ce2002}, \cite{ce2008}, Bennett, Carbery, Christ, and Tao \cite{bcct2008},
Bennet, Carbery, and Wright \cite{bcw2005}, Erdo{\v{g}}an and R. Oberlin \cite{eo2008}, and many others.  Earlier approaches to $L^p$-improving estimates for averaging operators, beginning with Littman \cite{littman1971}, Phong and Stein \cite{ps1986II}, \cite{ps1991}, \cite{ps1994}, \cite{ps1997}, including Greenleaf and Seeger \cite{gs1994}, \cite{gs1998}, Seeger \cite{seeger1993}, \cite{seeger1998}, and D. Oberlin \cite{oberlin1987}, \cite{oberlin1997}, \cite{oberlin1999}, have typically been based on oscillatory integral estimates which will not appear here.  
%% Melrose and Taylor \cite{mt1985}

\section{Estimation of integrals by pointwise values}

To begin this section, a number of definitions are in order.  First, suppose $K$ is a closed set contained in an open interval $I$.  For each nonnegative integer $n$, a function $f \in C^{n}(I)$ is said to be of polynomial type $n$ on $(K,I)$ when $f^{(n)}$ does not change sign (i.e., is nonnegative or nonpositive) and there exists a finite constant $C$ for which $\sup_{t \in I} |f^{(n)}(t)| \leq C \inf_{t \in K} |f^{(n)}(t)|$.  Any polynomial of degree $n$ on $I$ is, of course, polynomial type $n$ on $(K,I)$ for any closed set $K$ contained in $I$.

In general, if a regular probability measure $\mu$ is supported on the closed set $K$, it will necessary to consider functions which are of polynomial type on $(K_\epsilon,I)$ for some set $K_\epsilon$ slightly larger than $K$ (since, if $K$ has Lebesgue measure zero, the values of $f$ on $K$ are largely independent of the values of the derivatives of $f$ on the same set $K$).  To that end, given any closed set $K$ and any $\epsilon$, let $K_\epsilon$ be the union of $K$ and the sets $K_L$ and $K_R$ given by 
\begin{align*}
 K_R & := \set{t \in I}{ \inf \set{d \geq 0}{t + d \in K} \leq \epsilon \inf \set{d \geq 0}{t - d \in K}}, \\
K_L & := \set{t \in I}{ \inf \set{d \geq 0}{t - d \in K} \leq \epsilon \inf \set{d \geq 0}{t + d \in K} }
\end{align*}
($K_R$ and $K_L$ are the points which are bounded on both sides by the set $K$ but are proportionately much closer to $K$ on one side than the other; note that the Lebesgue measures of $K_L$ and $K_R$ are bounded by $\epsilon |I \setminus K|$).

The final definition needed to begin this section is a notion of the length of the set on which $\mu$ is supported: given a regular probability measure $\mu$ supported on an open interval $I \subset \R$, a positive integer $n$, and an $\epsilon \in (0,1)$, let $|\mu|_{n,\epsilon}$ be the infimum of $\sum_{j=1}^n |I_j|$ over all collections of closed intervals $\{I_1,\ldots,I_n\}$ which satisfy $\mu(\bigcup_{j=1}^n I_j) \geq 1-\epsilon$.
%\[ |\mu|_{n,\epsilon} := \inf \set{ \sum_{j=1}^n |I_j|}{ \mu\left( \bigcup_{j=1}^n I_j  \right) \geq 1-\epsilon \mbox{ and } I_j \mbox { is closed } \forall \ j=1,\ldots,n}. \]
Note that $|\mu|_{n,\epsilon}$ is decreasing in $n$, increasing in $\epsilon$, and $|\mu_K|_{n,\epsilon} \geq (1-\epsilon)|K|$ when, for example, $\mu_K$ is normalized Lebesgue measure on the set $K$.

Finally, a remark concerning notation is in order.  The two parameters having already appeared, namely $n$ and $\epsilon$, appear in essentially every inequality to come; in particular, most proportionality constants will vary as these parameters vary.  When the nature of these constants is uninteresting or otherwise considered unimportant, the notation $A \lesssim B$ will be used to indicate that there is a constant $C_{n,\epsilon}$ completely determined by $\epsilon$ and $n$ such that $A \leq C_{n,\epsilon} B$.

The main theorems of this section can now be stated.  The first is as follows:
\begin{theorem}
Let $K$ be a closed set contained in an open interval $I \subset \R$ (possibly infinite).  For any $\epsilon \in (0,1)$, let $K_{\epsilon}$ be the union of $K$ with $K_R$ and $K_L$ where \label{theorem1}
  For any positive integer $n$, any regular probability measure $\mu$ supported on $K$, and any $f \in C^{n}(I)$ for which $f^{(n)}$ does not change sign, it must be true that
\[ \int |f(t)| d \mu(t) \gtrsim |\mu|_{n,\epsilon}^n \inf_{t \in K_\epsilon} |f^{(n)}(t)|. \]
\end{theorem}

Theorems of this type are not new; see, for example Carbery, Christ, and Wright \cite{ccw1999}, Arhipov, Karacuba and \v{C}ubarikov \cite{akc1979}, or Rogers \cite{rogers2005}.  The main new feature is the presence of $|\mu|_{n,\epsilon}$ on the right-hand side; in most previous cases $\mu$ is assumed to be the uniform measure on some set $K$ and $|\mu|_{n,\epsilon}$ is replaced by $|K|$.  To prove the full uniform estimate (theorem \ref{theorem2} and its corollaries), it is necessary to distinguish the length $|\mu|_{n,\epsilon}$ from the measure of the support of $\mu$.

The second theorem of this section establishes uniform integral estimates from below by a supremum of the function on a set $E$ which depends only on the class of functions to which $f$ belongs.  This is, of course, the most difficult task necessary to establish any result along the lines of theorem \ref{theorem0}:
\begin{theorem}
Suppose $\mu$ is a regular probability measure supported on some closed $K \subset I$; fix some positive integer $n$ and $\epsilon \in (0,1)$.  There exists a closed set $E \subset I$ with at most $n$ connected components for which $\mu(E) \geq 1 - \epsilon$ and \label{theorem2}
\[ \int |f(t)| d \mu(t) \gtrsim c_{n,\epsilon} C^{-1} \sup_{t \in E} |f(t)| \]
for any function $f$ which is of polynomial type $n$ on $(K_\epsilon,I)$ with constant $C$.
\end{theorem}

In particular, an immediate corollary of theorems \ref{theorem1} and \ref{theorem2} is the following:
\begin{corollary}
Given a regular probability measure $\mu$ supported on $K \subset I$, an $\epsilon \in (0,1)$ and a positive integer $n$, there exists a closed interval $I'$ (possibly a single point) with $\mu(I') \geq \frac{1-\epsilon}{n}$ such that, for any function $f$ which is polynomial type $n$ on $(K_\epsilon,I)$ with constant $C$, \label{maincorollary}
\[ \int |f(t)| d \mu(t) \gtrsim C^{-1} \min\{ |I'|^j, |\mu|_{n,\epsilon}^j \} \sup_{t \in I'} |f^{(j)}(t)| \]
for any $j=0,\ldots,n$.
\end{corollary}
Theorem \ref{theorem0} from the introduction follows from this corollary when $f$ is polynomial of degree $n$ and $\mu$ is the uniform measure on $K$.

\subsection{Combinatorial considerations}

In what follows, let $V_n$ be the Vandermonde polynomial in $n$ variables, i.e., $V_n(t_1,\ldots,t_n) := \prod_{j > i} (t_j - t_i)$.  

The proof begins with a more detailed look at a standard idea:  the estimation of higher derivatives of a function $f$ via sampling at a finite number of points.  This is typically carried out using Lagrange interpolating polynomials.  The technique is the same here; the main difference is that there is that the structure of the ``remainder term'' is explored as well:
\begin{proposition}
Let $t_1, \ldots,t_{n+1}$ be distinct points in some interval $I$.  Given these points, there exists a nonnegative function $\psi_t(s)$, supported on $[t_1,t_{n+1}]$, with total integral $1$ such that \label{averageprop}
\begin{equation} \sum_{i=1}^{n+1} (-1)^{n+1-i} f(t_i) \frac{V_n(t_1,\ldots,\hat{t_i},\ldots,t_{n+1})}{V_{n+1}(t_1,\ldots,t_{n+1})} = \frac{1}{n!}  \int f^{(n)}(s) \psi_t(s) ds \label{averages}
\end{equation}
for any $f \in C^{(n)}(I)$ (here $\hat{t_i}$ indicates that $t_i$ is omitted).  This function $\psi_t$ has the following properties:  $(1)$ $\psi_t$ is supported on the convex hull of the $t_j$'s and has integral $1$, $(2)$ $\psi_t$ is a polynomial of degree at most $n-1$ on each interval containing none of the $t_j$'s, and $(3)$ $\psi_t \in C^{(n-2)}(I)$ when $n \geq 2$.
\end{proposition}
\begin{proof}
Let $a \in I$ be such that $a \leq \min_j t_j$.  If $f$ is any $n$-times continuously differentiable function on $I$, let
\[ g(t) := \int_a^t \frac{(t-s)^{n-1}}{(n-1)!} f^{(n)}(s) ds. \]
It is straightforward to check that $g$ is also $n$-times differentiable when $t > a$ and $g^{(n)}(t) = f^{(n)}(t)$ there ($g$ is nothing more than the remainder term for the degree $n-1$ Taylor polynomial at $a$).  This implies that $f-g$ is a polynomial of degree $n-1$ (the Taylor polynomial at $a$);  therefore it must be the case that the determinant
\begin{equation}
\left| \begin{array}{ccccc}
f(t_1) - g(t_1) & 1 & t_1 & \cdots & t_1^{n-1} \\
f(t_2) - g(t_2) & 1 & t_2 & \cdots & t_2^{n-1} \\
\vdots & \vdots & \vdots & \ddots & \vdots \\
f(t_{n+1}) - g(t_{n+1}) & 1 & t_{n+1} & \cdots & t_{n+1}^{n-1} 
\end{array}
\right| \label{zerodet1}
\end{equation}
vanishes.  Using Cramer's rule and the fact that the minors corresponding to entries in the first column are Vandermonde matrices, it follows that
\[ \sum_{i=1}^{n+1} (-1)^{i}(f(t_i) - g(t_i)) V_n(t_1,\ldots,\hat{t_i},\ldots,t_{n+1})  = 0. \]
since the $n$-th derivative of the difference vanishes at every point of $[t_1,t_{n+1}]$.  But this implies that
\begin{align*}
\sum_{i=1}^{n+1} & (-1)^{i}f(t_i) V_n(t_1,\ldots,\hat{t_i},\ldots,t_{n+1})  =
\sum_{i=1}^{n+1} (-1)^{i}g(t_i) V_n(t_1,\ldots,\hat{t_i},\ldots,t_{n+1}) \\
& = \int_a^\infty f^{(n)}(s) \frac{1}{(n-1)!} \sum_{t_i > s} (-1)^i (t_i - s)^{n-1} V_{n} (t_1,\ldots, \hat{t_i},\ldots, t_{n+1}) ds. %\frac{(t_i-s)^{n-1}}{\prod_{j \neq i} (t_i - t_j)} ds. 
\end{align*}
When $s \geq \max_j t_j$, the sum inside the integral vanishes because no terms are included.  When $s \leq \min_j t_j$, on the other hand, the sum again vanishes since it represents the determinant of a Vandermonde-type matrix whose first column is given by $(t_i - s)^{n-1}$ (which vanishes just like \eqref{zerodet1}).  To compute the integral of this function, it suffices to plug in $f(s) := \frac{s^n}{n!}$
:
\[ \sum_{i=1}^{n+1} (-1)^{i}\frac{t_i^n}{n!} V_n(t_1,\ldots,\hat{t_i},\ldots,t_{n+1}) = \frac{(-1)^{n+1}}{n!} V_{n+1}(t_1,\ldots,t_{n+1}). \]
Thus, one is led to define
\begin{align*}
\psi_t(s) := & \frac{n (-1)^{n+1}}{V_{n+1}(t_1,\ldots,t_{n+1})}  \sum_{t_i > s} (-1)^i (t_i - s)^{n-1} V_{n} (t_1,\ldots, \hat{t_i},\ldots, t_{n+1}) \\
& = \sum_{t_i > s} \frac{n (t_i - s)^{n-1}}{\prod_{j \neq i} (t_i - t_j)}.
\end{align*}
This $\psi_t$ has integral one and satisfies the correct integral identity.  Furthermore, it follows directly from this definition that $\psi_t$ is piecewise a polynomial of degree at most $n-1$ on all intervals not containing any $t_j$'s.  It is also immediate that $\psi_t \in C^{(n-2)}(I)$ when $n \geq 2$ because it is a finite linear combination of such functions (i.e., functions equal to $(t_i-s)^{n-1}$ when $s \leq t_i$ and equal to $0$ otherwise).  It remains to show that $\psi_t$ is nonnegative.  If this were not the case, it would be possible to find a function $f \in C^{(n)}(I)$ such that $f^{(n)}$ is strictly positive on $I$ but the right-hand side of \eqref{averages} is zero.  Examining \eqref{zerodet1}, this is possible only when there is a polynomial of degree $n-1$ which agrees with this function $f$ at $t_1,\ldots,t_{n+1}$.  Repeated applications of Rolle's theorem shows that this cannot be the case (i.e., the $n$-th derivative of $f$ must vanish at some point if $f$ agrees with a polynomial of degree $n-1$ at $n+1$ distinct points).
\end{proof}

It is perhaps worth noting that the three properties of $\psi_t$ (support and normalization, piecewise polynomial of degree $n-1$, and global $C^{(n-2)}$ regularity) uniquely determine $\psi_t$.  Even when the normalization condition is dropped, there is still only a one-dimensional family of such functions, namely, multiples of $\psi_t$; this is to say that there are no nontrivial piecwise functions satisfying these conditions and having integral $0$.  The proof of this fact proceeds inductively on $n$ in a fairly standard way.

The advantage gained in finding $\psi_t$ more-or-less explicitly is that it allows one to improve upon the trivial estimate from below on $\int f^{(n)}(s) \psi_t(s) ds$ to exploit the fact that, while $\psi_t$ is not supported at the points $t_j$, it must always, in fact, have some positive proportion of its mass which lies near the $t_j$'s.  In other words, if the $t_j$'s happen to be separated by some large distance, it never occurs that an overwhelming fraction of the mass of $\psi_t$ is concentrated inside that gap.  This fact will be made precise during the proof of theorem \ref{theorem1}.

%One now comes to the main piece of machinery:  Given a function $f$ which is $n$-times differentiable, the integral $\int |f(t)| d \mu(t)$ may be estimated by moving to a multilinear setting and using the fact that information about the derivatives of $f$ give information about the possible values of $f$ at an ensemble of points.
\begin{proposition}
For any regular probability measure $\mu$ on an interval $I \subset \R$ and any positive integer $n$, let \label{intervals}
\[ (\ell_n(\mu))^n  := \frac{\int |V_{n+1}(t_1,\ldots,t_{n+1})| d \mu(t_1) \cdots d \mu(t_{n+1})}{\int |V_n(t_1,\ldots,t_n)| d \mu(t_1) \cdots d \mu(t_n)}.\]
The quantity $\ell_n(\mu)$ is zero if and only if $\mu$ is supported on a set of $n$ or fewer points.  Furthermore, given any $\epsilon > 0$, there exists a finite collection of at most $n$ closed, disjoint intervals $I_j$ (possibly length zero) such that $\mu( \bigcup_j I_j ) \geq 1 - \epsilon$, $| \bigcup_j I_j | \lesssim \ell_n(\mu)$ (that is, the Lebesgue measure of the union) and $\mu(I_j) \gtrsim 1$. These intervals will be called the $(n,\epsilon)$-children of $\mu$.  %(that is, the $\mu$ measure of $I_j$ is bounded below by a constant depending only on $n$ and $\epsilon$).
\end{proposition}
\begin{proof}
First of all, it is necessarily true that $\ell_n(\mu) = 0$ if the mass of $\mu$ is supported on a finite set of $n$ or fewer points, since in this case the Vandermonde polynomial $V_{n+1}$ vanishes almost everywhere on the $(n+1)-$fold product of $\mu$.  In all other cases, the distribution function $\mu((-\infty, t] \cap I)$ must take at least $n+1$ distinct, nonzero values, meaning that the interval $I$ may be partitioned into at least $n+1$ disjoint pieces, each of which has nonzero $\mu$-measure.  Since $V_{n+1}$ does not vanish when each $t_i$ belongs to a distinct element of the partition, the integral cannot be zero.

Assuming now that $\mu$ is not supported on a set of $n$ points, consider the ratio
\[ \frac{|V_{n+1}(t_1,\ldots,t_{n+1})|}{|V_n(t_1,\ldots,t_n)|} = \prod_{i=1}^n |t_n - t_i|. \]
If $E_{t_1,\ldots,t_n}$ is the set $\set{ t_{n+1}}{ \prod_{i=1}^n |t_n - t_i| \leq 2 \epsilon^{-1} (\ell_n(\mu))^n}$, it follows that
\[ \int \frac{|V_{n+1}(t_1,\ldots,t_{n+1})|}{|V_n(t_1,\ldots,t_n)|} d \mu(t_{n+1}) \geq (1 - \mu(E_{t_1,\ldots,t_n})) 2 \epsilon^{-1} (\ell_n(\mu))^n
\]
for each possible ensemble $t_1,\ldots,t_n$ of distinct points.  Multiplying both sides by $|V_{n}(t_1,\ldots,t_n)|$ and integrating $d \mu(t_1) \cdots d \mu(t_n)$, it follows that
\[ (\ell_n(\mu))^n \geq \inf_{t_1,\ldots,t_n \in I} \epsilon^{-1} (1 - \mu(E_{t_1,\ldots,t_n})) (\ell_n(\mu))^n \]
so there must be a choice of $t_1,\ldots,t_n$ for which $\mu(E_{t_1,\ldots,t_n}) \geq 1- \frac{\epsilon}{2}$.  This sublevel set consists of at most $n$ closed connected components.  The Lebesgue measure of the sublevel set is at most $2n \ell_n(\mu)$ (since the sublevel condition requires, in particular, that $t_{n+1}$ must be within distance $\ell_n(\mu)$ of at least one of the $t_j$'s for $j=1,\ldots,n$).  Let the $(n,\epsilon)$-children of $\mu$ be the connected components of $E_{t_1,\ldots,t_n}$ whose $\mu$ measure is at least $\frac{\epsilon}{2n}$.  Clearly they are disjoint, have bounded lengths, and the $\mu$ measure of the union is at least $1 - \epsilon$.
\end{proof}

\begin{proposition}
Let $t_1,\ldots,t_{N}$ be points in some interval $I$.  For each positive integer $n$, there is a closed set $E_n \subset I$ which consists of no more than $n$ connected components, contains $t_j$ for $j=1,\ldots,N$ and satisfies \label{closegaps}
\[ \sup_{t \in E_n} |f(t)| \leq (n+1) 2^n \max_{j=1,\ldots,N} |f(t_j)| \]
for any $f \in C^{n}(I)$ whose $n$-th derivative does not change sign.
%Suppose $t_1, t_2, \ldots, t_{n+1}$ are points in some interval $I$.  There is a closed interval $I'$ containing at least two of these points such that
%\[ \sup_{t \in I'} |f(t)| \leq 2^{n} \sum_{j=1}^{n+1} |f(t_j)| \]
%for any $f \in C^{n}(I)$ for which $f^{(n)}$ does not change sign (i.e., is either nonnegative or nonpositive).
\end{proposition}
\begin{proof}
Without loss of generality, it suffices to assume that $t_1 < t_2 < \cdots < t_N$ and that $N \geq n+1$.  It is also permissible to assume that $f^{(n)}$ is nonnegative.  

Consider first the case $N = n+1$.  Fix $t_1 < t_2 < \cdots < t_{n+1}$ and fix $k$ to be the index which maximizes $V_{n}(t_1,\ldots, \hat{t_k}, \ldots, t_{n+1})$; notice that the index can never equal $1$ or $n+1$ (since omitting $t_2$ or $t_{n}$, respectively, will always increase the product).  Let $I'$ be the shorter interval of $[t_{k-1},t_k]$ or $[t_{k},t_{k+1}]$ (if they have the same length, either choice is acceptable).  
Given $f \in C^{(n)}(I)$, let
\[ \tilde f(t) := f(t) - \sum_{j \neq k} f(t_j) \prod_{i \neq j,k} \frac{t - t_i}{t_j - t_i}. \]
Clearly $\tilde f(t_j) = 0$ for $j \neq k$.  It suffices to work with $\tilde f$ rather than $f$ since the values of $f$ and $\tilde f$ do not differ appreciably on $I'$ or at $t_k$.  More precisely, at $t=t_k$ one has
\[ |\tilde f(t_k)| \leq |f(t_k)| + \sum_{j \neq k} |f(t_k)| \frac{|V_n(t_1,\ldots,\hat{t_j},\ldots,t_{n+1})|}{|V_n(t_1,\ldots,\hat{t_k},\ldots,t_{n+1})|} \leq (n+1) \max_j |f(t_j)|. \]
Next, if $t \in I'$, then $\prod_{i \neq j,k} |t - t_i| \leq 2^{n-1} \prod_{i \neq j,k} |t_k - t_i|$ since $|t - t_i| \leq |t - t_k| + |t_k - t_i| \leq \min_j \{ |t_k - t_j| \} + |t_k - t_i|$.  It must therefore be the case that
\[|f(t)| \leq |\tilde f(t)| + n 2^{n-1} \max_j |f(t_j)|; \]
hence if $|\tilde f(t)| \leq C \max_j |\tilde f(t_j)|$, then $|f(t)| \leq (C(n+1) + n 2^{n-1}) \max_j |f(t_j)|$.

Now \eqref{averages} implies that, for $t \in I'$, $(-1)^{n+1-k} f(t) \geq 0$ when $f^{(n)}$ is nonnegative on $I$ by virtue of the fact that the Vandermonde polynomials are positive and $\tilde f(t_i)$ vanishes for $i \neq k$.  Fix some $t \in I'$; let $t_1' < t_2' < \cdots < t_{n+1}'$ be the sequence of numbers obtained by replacing $t_{k-1}$ with $t$ when $I'=[t_{k-1},t_k]$ or replacing $t_{k+1}$ with $t$ when $I' = [t_{k},t_{k+1}]$ (so that $t_j' = t_j$ for all but one value of $j$).  For this collection of points, \eqref{averages} implies that
\[ \sum_{j=1}^{n+1} (-1)^{n+1-j} \tilde f(t_j') \frac{V_n(t_1',\ldots,\hat{t_j'},\ldots,t_{n+1}')}{V_{n+1}(t_1',\ldots,t_{n+1}')} \geq 0 \]
as well.  Here all but two terms must vanish (by virtue of the vanishing of $\tilde f$).  Thus
\[ (-1)^{n+1-k} \tilde f(t_k) V_n(t_1',\ldots,\hat{t_k'},\ldots,t_{n+1}') \geq (-1)^{n+1-k} \tilde f(t) V_n (t_1',\ldots, \hat{t}, \ldots, t_{n+1}') \]
(where $\hat{t}$ is properly interpreted as $\widehat{t_{k-1}'}$ or $\widehat{t_{k+1}'}$ depending on which of those indices had its corresponding value replaced by $t$).  Both sides are nonnegative, and just as before, the Vandermonde polynomial on the left-hand side is at most a factor of $2^{n-1}$ larger than the corresponding polynomial on the right-hand side.  It must therefore be the case that $|\tilde f(t)| \leq 2^{n-1} |\tilde f(t_k)|$ and hence $|f(t)| \leq (n+1) 2^n \max_j |f(t_j)|$.

For the case of general $N$, suppose that there were more than $n$ intervals of the form $[t_k,t_{k+1}]$ for which
\[ \sup_{t \in [t_j,t_{j+1}]} |f(t)| \geq (n+1) 2^n \max_{j=1,\ldots,N} |f(t_j)|. \]
If $s_1,\ldots,s_{n+1}$ are the leftmost endpoints of such intervals, then there must exist an interval $I'$ on which
\[ \sup_{t \in I'} |f(t)| \leq (n+1) 2^n \max_{j=1,\ldots,n+1} |f(s_j)| \]
for any function $f$ whose $n$-th derivative does not change sign.  The leftmost endpoint of this interval coincides with one of the $s_j$'s, hence $I'$ must contain $[t_j,t_{j+1}]$, giving a contradiction.
\end{proof}

\subsection{The proofs of theorems \ref{theorem1} and \ref{theorem2}}

\begin{proof}[Proof of theorem \ref{theorem1}.]
The proof proceeds in two parts; first for any $f \in C^{n}(I)$,
\begin{equation}
 \int |f(t)| d \mu(t) \geq \frac{(\ell_n(\mu))^n}{(n+1)!} \inf_{t \in I} |f^{(n)}(t)|. \label{mineq1}
\end{equation}
Following this the general situation is considered:  if the support of $\mu$ is contained in some closed set $K$ and $K_\epsilon \supset K$ is as defined in theorem \ref{theorem1} then
\begin{equation}
 \int |f(t)| d \mu(t) \gtrsim (\ell_n(\mu))^n \inf_{t \in K_\epsilon} f^{(n)}(t) \label{mineq2}
\end{equation}
provided that $f^{(n)}$ does not change sign on $I$.

Regarding \eqref{mineq1}, for any function $f$ one has the trivial inequality
\begin{align*}
(n+1) & \int |f(t)| d \mu(t) \int |V_n(t_1,\ldots,t_n)| d \mu(t_1) \cdots d \mu(t_n)  
\\ 
& \geq \int \left| \sum_{i=1}^{n+1} (-1)^{n+1-i} f(t_i) V_n(t_1,\ldots,\hat{t_i},\ldots,t_{n+1}) \right| d \mu(t_1) \cdots d \mu(t_{n+1}).
\end{align*}
If one supposes further that $f \in C^{(n)}(I)$, equation \eqref{averages} from proposition \ref{averageprop} gives that
\begin{align*}
(n+1) & \int |f(t)| d \mu(t) \int  |V_n(t_1,\ldots,t_n)| d \mu(t_1) \cdots d \mu(t_n)  
\\ 
& \geq \frac{1}{n!} \int \left|\int f^{(n)}(s) \psi_t(s) ds \right| |V_{n+1}(t_1,\ldots,t_{n+1})| d \mu(t_1) \cdots d \mu(t_{n+1}).
\end{align*}
Since $|\int f^{(n)}(s) \psi_t(s) ds| \geq \inf_{t \in S} |f^{(n)}(t)|$, the inequality \eqref{mineq1} must be true.

The proof of \eqref{mineq2} and, hence, theorem \ref{theorem1} follows from a closer estimation of $\int f^{(n)}(s) \psi_t(s) ds$. Let $t_1,\ldots,t_{n+1}$ be taken from some closed set $K \subset I$.  Fix any $\epsilon > 0$ and let $K_\epsilon$ (as in theorem \ref{theorem1}) be the closed set of points $t$ for which $t_{+} := \inf \set{s \in K}{s \geq t}$ and $t_{-} := \sup \set{s \in K}{s \leq t}$ both exist and satisfy either $|t_+ - t| \leq \epsilon |t_+ - t_{-}|$ or $|t_{-} - t| \leq \epsilon|t_+ - t_{-}|$.  It suffices to show that, for any continuous function $g$ on $I$ which does not change sign,
\[ \int g(s) \psi_t(s) ds \gtrsim \inf_{s \in K_\epsilon} g(s). \]
To that end, let $p_c(t) := \sum_{i=0}^{n-1} c_i t^i$ where the $c_i$ are real coefficients whose squares sum to $1$.  The ratio
\[ \frac{\int_0^{\epsilon} |p_c(t)| dt + \int_{1-\epsilon}^\epsilon |p_c(t)| dt}{\int_{0}^1 |p_c(t)| dt} \]
is never equal to zero for any polynomial $p$ of degree at most $n-1$; therefore compactness of the unit sphere and homogeneity imply that there exists a constant $C_{n,\epsilon}$ such that $\int_0^\epsilon  |p(t)| dt + \int_{1-\epsilon}^1 |p(t)| dt \geq C_{n,\epsilon} \int_{0}^1 |p(t)| dt$ for any polynomial $p$ of degree $n-1$.  By a suitable change of variables, one has the integral of $|p|$ over the ends (each of length $\epsilon$ times the length of the whole interval) of any interval is bounded below by a constant times the integral over the whole interval.

Consider now the integral of $g$ against $\psi_t$.  Clearly there exist a countable number of open, disjoint intervals $I_j$ in the convex hull of $K$ such that
\[ \int g(s) \psi_t(s) ds = \int_K g(s) \psi_t(s) ds + \sum_j \int_{I_j} g(s) \psi_t(s) ds. \]
Since $K \subset K_\epsilon$, $\int_K g(s) \psi_t(s) ds \geq (\inf_{s \in K_\epsilon} g(s)) \int_K \psi_t(s) ds$.  As for each $I_j$, the ends of these intervals are in $K_\epsilon$ as well (and $g$ is nonnegative on the interior region of $I_j$).  Furthermore, $\psi_t$ is a polynomial of degree at most $n-1$ on $I_j$ since this interval contains no points of $K$ (hence none of the $t_j$'s).  Thus
\[ \int g(s) \psi_t(s) ds \geq \inf_{s \in K_\epsilon} g(s) \left( \int_K \psi_t(s) ds + C_{n,\epsilon} \sum_j \int_{I_j} \psi_t(s) ds \right). \]
Summing these finishes the proof.
\end{proof}

\begin{proof}[Proof of theorem \ref{theorem2}.]
Fix some $\epsilon'$ and positive integer $n$.  Let ${\cal C}^{(1)}$ be the collection of $(n,\epsilon')$-children of $\mu$ given by proposition \ref{intervals}.  Next let ${\cal C}^{(2)}$ be the collection of all $(n-1,\epsilon')$-children of intervals $I \in {\cal C}^{(1)}$, where the children of an interval $I$ are understood as the children of the measure $\mu_I := \mu(I)^{-1} \left. \mu \right|_I$ (note that this will always be well-defined since the $\mu$-measures of children always have a minimal amount of mass as controlled by $\epsilon'$).  Continue in this manner until the collection ${\cal C}^{(n)}$ (the $1$-children of the collection ${\cal C}^{(n-1)}$) is obtained.  

Suppose that $f$ is of polynomial type $n$ on $(K_\epsilon,I)$ with constant $C$.  This implies by (the proof of) theorem \ref{theorem1}, that
\begin{equation}
\int |f(t)| d \mu(t) \gtrsim C^{-1} (\ell_n(\mu))^n \sup_{t \in I} |f^{(n)}(t)|. \label{step0}
\end{equation}
 Now, given some interval $I_0 \in {\cal C}^{(n)}$, let $I_j$ be the unique element of ${\cal C}^{(n-j)}$ containing $I_1$.  For convenience, let $I_{-1} := I_0$ and $I_n := I$ (the interval on which $\mu$ is supported).  Fix $c^{-1} = \log \frac{3}{2}$, and choose the $j$ in $0,\ldots,n$ which maximizes
\begin{equation}
 c^j |I_{j-1}|^j \sup_{t \in I_{j}} |f^{(j)}(t)| \label{ptsup}
\end{equation}
(if $j$ is not unique, choose the largest such $j$).  If $j=n$, then inequality \eqref{step0} has as an immediate consequence that
\[ \int |f(t)| d \mu(t) \gtrsim C^{-1} |I_{j-1}|^j \sup_{t \in I_j} |f^{(j)}(t)| \]
for $j=0,\ldots,n$.  Suppose instead that the maximizing index $j$ is not equal to $n$.  In this case, let $s_0 \in I_{j}$ be the point where the supremum is obtained.  It follows that, for any $s \in I_{j}$, 
\begin{align*}
c^j |I_{j-1}|^j |f^{(j)}(s) - f^{(j)}(s_0)| & \leq \sum_{k=j+1}^{n} \frac{c^j |I_{j-1}|^j |s-s_0|^{k-j}}{(k-j)!} \sup_{t \in I_{k}} |f^{(k)}(t) | \\ 
& \leq \sum_{k=j+1}^n \frac{c^{j-k}}{(k-j)!} c^k |I_{k-1}|^k \sup_{t \in I_{k}} |f^{(k)}(t)| \\
& \leq (e^\frac{1}{c} - 1) c^j |I_{j-1}|^j \sup_{ t \in I_{j}} |f^{(j)}(t)|.
\end{align*}
It must therefore be the case that $\sup_{t \in I_{j}} |f^{(j)}(t)| \leq 2 \inf_{t \in I_{j}} |f^{(j)}(t)|$, meaning that $f$ is of polynomial type $j$ on $(I_j,I_j)$ with constant $2$.  Therefore theorem \ref{theorem1} and proposition \ref{intervals} guarantee that
\begin{align*}
 \int_{I_{j}} |f(t)| d \mu(t) & \gtrsim \mu(I_{j}) (\ell_j( \mu_{I_j}))^j \sup_{t \in I_{j}} |f^{(j)}(t)| \\
& \gtrsim  |I_{j-1}|^j \sup_{t \in I_{j}} |f^{(j)}(t)| \\
& \gtrsim  |I_{k-1}|^k \sup_{t \in I_k} |f^{(k)}(t)| \ \forall k = 0,\ldots,n.
\end{align*}
In particular, there is now a collection of intervals $I'$, namely ${\cal C}^{(n)}$ such that the integral $\int |f| d \mu(t) \gtrsim C^{-1} \sup_{t \in I'} |f(t)|$ for any function $f$ which is polynomial type $n$ on $(K_\epsilon,I)$ with constant $C$. By proposition \ref{closegaps}, these intervals can be joined together independently of $f$ so that they cover the same set as before but consist of no more than $n$ connected components.  In particular, the $\mu$-measure is at least $(1-\epsilon')^n$, which can be made greater than $1 - \epsilon$ for suitably-chosen $\epsilon'$.
\end{proof}

As for the remaining corollary, choose the interval $I'$ to be a connected component of $E$ as given by theorem \ref{theorem2} which has $\mu$-measure at least $\frac{1-\epsilon}{n}$.  In this case, the conclusion of the corollary follows immediately from the combined conclusions of theorems \ref{theorem1} and \ref{theorem2}:
\begin{proposition}
Suppose that $f \in C^{n}(I')$.  For any $j=0,\ldots,n$,
\[ \min\{ |I'|^j, \ell^j \}  \sup_{t \in I'} |f^{(j)}(t)| \lesssim \sup_{t \in I'} |f(t)| +  \ell^j \sup_{t \in I'} |f^{(n)}(t)|. \]
\end{proposition}
\begin{proof}
As in the proof of theorem \ref{theorem2}, let $j$ be the index out of $0,\ldots,n$ which maximizes $2^j \min\{ |I|^j, \ell^j \}  \sup_{t \in I} |f^{(j)}(t)|$.  If $j=n$, then there is nothing else to prove.  Otherwise, for any $t,s \in I$, the mean-value theorem assures that
\[ |f^{(j)}(t) - f^{(j)}(s)| \leq |t-s| \sup_{u \in I'} |f^{(j+1)}(u)|. \]
Provided that $|t-s| \leq \min \{|I|,\ell \}$, the right-hand side is bounded above by $\frac{1}{2} \sup_{u \in I'} |f^{(j)}(u)|$.  In particular, if $I''$ is any interval of length $\min \{|I'|,\ell \}$ containing the point where $f^{(j)}$ achieves its maximum, then the supremum of $f^{(j)}$ on that interval is bounded by twice the infimum.  But in this case equation \eqref{averages} guarantees that $\min\{ |I'|^j, \ell^j \}  \sup_{t \in I'} |f^{(j)}(t)|$ is bounded below by a constant (depending only on $n$) times the supremum of $f$ (to apply equation \eqref{averages}, simply choose evenly-spaced points of $I''$).
\end{proof}

\section{$L^p$-improving estimates for polynomial curves}

This section is devoted to the proof of theorem \ref{averageop}.  The proof is itself divided into two parts.  The first is the main argument, relying on integral estimates and refinements (and, in particular, relying on theorem \ref{theorem0}).  With the one-dimensional integral estimates already established, the main portion of the proof of theorem \ref{averageop} is remarkably short.  The second part of the proof deals with counting solutions of the iterated flows of $X_1$ and $X_2$.  As is customary, this boils down to an application of B\'{e}zout's theorem; the difference here is that the vector fields $X_1$ and $X_2$ must first be lifted to a nilpotent Lie group (as was done by Christ, Nagel, Stein, and Wainger \cite{cnsw1999}) to produce a setting in which the flows correspond to polynomial mappings.

\subsection{Refinements, re-centering, and integral estimates}

The main innovation of the work of Tao and Wright over the original paper of Christ was the observation that, under certain circumstances, the integral of a function over the flow $\exp(tX)(x_0)$ of a vector field $X$ can be estimated from below by the value of that function or its derivatives evaluated at $t=0$ (the ``central'' part of central sets of a fixed width).  Of course, it is not always possible to make an estimate of this sort (that is, it is easy to construct examples of functions which happen for particular $x_0$ to be much larger at $t=0$ than at the other values of $t$ which form the support of the integral).  Tao and Wright circumvent this problem by introducing the notion of a set with width $w$; more recently, Christ \cite{christ2006} avoids this problem by introducing $(\epsilon,\delta)$-generic sets.  The problem with the construction of Tao and Wright is that, in the process, unavoidable small losses are encountered in various exponents which lead to less-than-sharp restricted weak-type results.  One way to avoid this problem, at least in the case of polynomial curves, is to use theorem \ref{theorem0} instead of introducing central sets of fixed width.  The application of theorem \ref{theorem0} comes in the following lemma which describes the set of $x_0$'s for which this re-centering can be accomplished.  This new set is called a refinement of the original:
\begin{proposition}
Let $U' \subset \R^{d+1}$ be open and $\pi : U \rightarrow \R^d$ have surjective differential at every point; let $X$ be a nonvanishing vector field on $U'$ for which $d \pi (X) = 0$.  Let $U \subset U'$ be open and bounded and fix a positive integer $n$.  There exists a nonzero constant $c$ depending on $n$ and the bounded subset $U$ such that, for any measurable $\Omega \subset U$, there is a refinement $\Omega' \subset \Omega$ with $|\Omega'| \geq c |\Omega|$ such that, for any $x_0 \in \Omega'$, the integral estimate
\begin{align*}
 \int |f(t,x_0)| & \chi_{\Omega} ( \exp(tX)(x_0)) dt \geq \\
& c \max_{j=0,\ldots,n} \left\{ \left( \frac{|\Omega|}{|\pi(\Omega)|} \right)^{j+1} \left| \frac{\partial^j f}{\partial t^j} (0,x_0) \right| \right\}
\end{align*}
holds for any function $f(t,x_0)$ which is a \label{refinement} polynomial of degree at most $n$ for each fixed $x_0$.
\end{proposition}
\begin{proof}
It suffices to restrict attention to the portion of $\Omega$ which lies on a particular integral curve of $X$ and show that a positive proportion of such points can be taken to lie in $\Omega'$.  In this case, one can change variables so that $X$ simply coincides with a coordinate direction and $\exp(tX)$ is simply translation by $t$ in that particular direction.  For a first approximation, $\Omega'$ is taken to be the set of all points $x_0$ such that 
\[ \int \chi_{\Omega}(\exp(tX)(x_0)) dt \geq c \frac{|\Omega|}{|\pi(\Omega)|} \]
for some small $c$; Fubini's theorem guarantees that the set $\Omega \setminus \Omega'$ is necessarily only a small fraction of the set $\Omega$.  Now the set $\Omega'$ as defined is still slightly too big.  However, restricting attention to the intersection of a fiber of $\pi$ with the set $\Omega'$,  it suffices to prove that, for any set $K \subset \R$, there is a subset $K' \subset K$ with $|K'| \geq c|K|$ for which
\begin{equation} \int |f(t,s)| \chi_{K}(t+s) dt \geq c \max_{j=0,\ldots,n} \left\{ |K|^{j+1} \left| \frac{\partial^j f}{\partial t^j} (0,s) \right| \right\} 
\label{intest}
\end{equation}
whenever $s \in K'$.  But this inequality follows directly from theorem \ref{theorem0} when, for example, $s$ lies inside the interval $I$ given by that theorem.  Thus, if $\Omega'$ is further reduced to contain only those points in each fiber which lie in the corresponding interval $I$ given by theorem \ref{theorem0}, the proposition follows.
\end{proof}
Notice that, since the refinement $\Omega'$ is contained in $\Omega$ and has $|\Omega'| \geq c |\Omega|$, it follows that
\[ \frac{|\Omega'|}{|\tilde{\pi}(\Omega')|} \geq c \frac{|\Omega|}{|\tilde \pi(\Omega)|} \]
for any projection $\tilde \pi$ (which may or may not be the same as the projection used for refining).  Thus when proposition \ref{refinement} is applied iteratively, it is always possible for the {\it original} $\Omega$ to appear on the right-hand side of \eqref{intest} at the price of a slightly worse constant (which is not a problem as long as the iterations terminate after a uniformly bounded number of steps).

The proof of theorem \ref{averageop} now proceeds exactly as in the work of Christ \cite{christ1998} or Tao and Wright \cite{tw2003}.  For each $x_0$, consider the mapping
\[ \Phi_{x_0} (t_1,\ldots,t_{d+1}) := \exp(t_1 X_1) \circ \cdots \circ \exp(t_{d+1} X_{d+1}) (x_0). \]
In the next section, it will be established that, for fixed $x_0$, this mapping has finite multiplicity everywhere except for some set of times $(t_1,\ldots,t_{d+1})$ which has $(d+1)$-dimensional Lebesgue measure zero.  Thus it follows that, for any (measurable) set $\Omega \subset U$,
\[ |\Omega| \geq c \int \chi_{\Omega} ( \Phi_{x_0}(t_1,\ldots,t_{d+1})) |J_{x_0}(t_1,\ldots,t_{d+1})| dt_1 \cdots dt_d \]
where $c$ is the reciprocal of the maximum multiplicity and $J_{x_0}(t)$ is the Jacobian determinant of the mapping $\Phi_{x_0}(t)$.  If $\Omega'$ is the refinement via the previous proposition with respect to the mapping $\pi_1$ and vector field $X_1$, it follows that
\begin{align*}
\int \chi_{\Omega} & (  \Phi_{x_0}(t_1,\ldots,t_{d+1}))  |J_{x_0}(t_1,\ldots,t_{d+1})| dt_1 \cdots dt_{d+1} \\
\geq  c' & \left( \frac{|\Omega|}{|\pi_1(\Omega)|} \right)^{j+1} \int \chi_{\Omega'} ( \Phi_{x_0}(0,t_2,\ldots,t_d)) \left| \frac{\partial^j J_{x_0}}{\partial t_1^j}(0,t_2,\ldots,t_d) \right| dt_2 \cdots d t_{d+1}
\end{align*}
for any $j=0,\ldots,n$.  But this new integral can, in turn, be estimated in exactly the same way by refining $\Omega'$ with respect to the mapping $\pi_2$ and the vector field $X_2$ and so on.  The end result is that, for any multiindex $\alpha$ there is a constant $c_{\alpha}$ such that 
\begin{equation} |\Omega| \geq c \prod_{i=1}^{d+1} \left( \frac{|\Omega|}{|\pi_i(\Omega)|} \right)^{\alpha_i+1} \left| \frac{\partial^\alpha J_{x_0} }{\partial t^\alpha}(0) \right| \label{bigest}
\end{equation}
for all $x_0$ in some iterated refinement of $\Omega$ (which, in particular, will have nonzero measure).  Notice, however, that when $\Omega$ is defined by taking $\chi_{\Omega}(x) := \chi_F(\pi_1(x)) \chi_{G} (\pi_2(x))$, this inequality may be manipulated to give theorem \ref{averageop}.

It is also worth noting that when equation \eqref{bigest} is summed over all multiindices $\alpha$, one obtains the rather interesting geometric inequality that 
\[ |\Omega| \geq c \left|B_0\left(x_0, \frac{|\Omega|}{|\pi_1(\Omega)|}, \frac{|\Omega|}{|\pi_2(\Omega)|} \right)\right| \]
where $B_0(x_0,\delta_1,\delta_2)$ is the image of the set $[-\delta_1,\delta_1] \times [-\delta_2 , \delta_2] \times \cdots \times [-\delta_{d+1},\delta_{d+1}]$ (with the usual periodicity convention) under the mapping $\Phi_{x_0}$.  The measure of this set is, in turn, comparable to the measure of the two-parameter Carnot-Carath\'{e}odory ball $B(x_0; \delta_1,\delta_2)$ of Tao and Wright.  Thus equation \eqref{bigest} gives a rather direct proof of the improved version of Tao and Wright's equation (66) mentioned in the second remark at the end of the paper.
\subsection{Lifting as related to polynomial curves}

In this section, it remains to show that $\Phi_{x_0}$ has bounded multiplicity outside some exceptional set and that the Jacobian determinant $J_{x_0}(t)$ is (up to a factor bounded away from $0$) a polynomial function of the $t$ parameters.  The main idea of the proof of these facts is a lifting argument involving the Baker-Campbell-Hausdorff formula.  The reader is referred to the paper of Christ, Nagel, Stein, and Wainger \cite{cnsw1999} for a thorough treatment of this topic.  In the proof at hand, this previous must be improved slightly (to obtain exact formulas rather than asymptotic ones), but there is not any added difficulty; in fact, it will suffice to only reproduce a few very small pieces of this much larger work.

To that end, let ${\cal N}$ be the collection of all words $w$ for which $X_w$ (the commutator as defined at the beginning of the paper) does not vanish identically.  For any $s \in \R^{{\cal N}}$, let
\[ s \cdot X := \sum_{w \in {\cal N}} s_w X_w. \]
Fix a bounded open set $U \subset \R^{d+1}$ on which $X_1$ and $X_2$ are defined and fix $x_0 \in U$.  Let $\tilde U \subset \R^{\cal N}$ be the collection of all $s$ for which $\exp( \theta s \cdot X)(x_0) \in U$ for all $\theta \in [0,1]$ (the inclusion of $\theta < 1$ guarantees that for any $s$, the associated integral curves used to define $\exp(s \cdot X)(x_0)$ remain in $U$).  Suppose for the moment that it is possible to establish the following two facts:
\begin{enumerate}
\item For any $x \in U$, the fiber $\exp(s \cdot X)(x_0) = x$ of the mapping $\exp(s \cdot X)(x_0)$ (as a function from $\tilde U$ to $U$) can be parametrized by a polynomial function, i.e., there exists a mapping $N_x(u)$ which parametrizes the fiber, has coordinate functions which are polynomials in the $u$ variables, and has surjective differential.
\item There exists a lifting $\tilde \Phi_{x_0}(t)$ of $\Phi_{x_0}(t)$ which is also polynomial, that is, $\exp ( \tilde \Phi_{x_0}(t) \cdot X)(x_0) = \Phi_{x_0}(t)$ and the coordinate functions of $\tilde \Phi_{x_0}$ are polynomial functions of $t$.
\end{enumerate}
These facts combined allow one to use B\'{e}zout's theorem, just as was employed in the original paper of Christ \cite{christ1998}.  Specifically, it follows from these facts that the equation $\Phi_{x_0}(t) = x$ has a solution for some $t$ only when the equations $\tilde \Phi_{x_0}(t) = N_x(u)$ have a solution for the same $t$ and some value of $u$.  In the usual manner, an additional parameter $v$ can be added to these equations in such a way that the resulting system of equations is homogeneous in $(t,u,v)$ and reduces to the original system when $v=1$.  Now B\'{e}zout's theorem guarantees that the number of irreducible components (in complex projective space) of the variety determined by these equations is at most the product of the degrees (and, in particular, does not depend on the particular choice of $x$).  See Fulton \cite{fulton1984}, chapter 8, section 4 (and, in particular, example 8.4.6) for this version of B\'{e}zout's theorem.  

Now The Jacobian determinant $J_{x_0}(t)$ is nonzero at $t_0$ only when the graph of $\tilde \Phi_{x_0}(t)$ is transverse to the fibers of $\exp(s \cdot X)(x_0)$ at the point $\tilde \Phi_{x_0}(t_0)$.  Thus solutions to $\Phi_{x_0}(t) = x$ (for real $t$) at which the Jacobian determinant $J_{x_0}(t)$ is nonvanishing arise only when there is an isolated solution of the system $\tilde \Phi_{x_0}(t) = N_x(u)$ in $(t,u)$ space; that is, when the Jacobian determinant with respect to $(t,u)$ of the mapping $\tilde \Phi_{x_0}(t) - N_x(u)$ is also nonzero.  This, in turn, guarantees that the solution $(t,u)$ remains isolated amongst complex solutions as well.  But any such isolated solution, in particular, corresponds to an irreducible component of the zero set of the homogeneous equations in complex projective space; it therefore follows that there is a uniform bound on the number of solutions to $\Phi_{x_0}(t) = x$ which occur where the Jacobian determinant is nonvanishing.  

As for any solutions at which the Jacobian determinant may vanish, Sard's lemma guarantees that the set of times $t$ at which the Jacobian determinant does vanish has ($d+1$-dimensional) measure zero.  In particular, this also means that the set of points $x \in U$ for which there can exist a solution to $\Phi_{x_0}(t) = x$ with vanishing Jacobian determinant is also a set of measure zero in $U$.  Thus, except for an exceptional set of $x$'s of measure zero, the system of equations $\Phi_{x_0}(t) = x$ has a uniformly bounded number of solutions, and the Jacobian determinant of $\Phi$ at each such solution is nonzero.  Thus the usual change-of-variables formula gives at once the desired inequality
\[ |\Omega| \geq c \int \chi_{\Omega}(\Phi_{x_0}(t)) |J_{x_0}(t)| dt \]
for any set $\Omega \subset U$.

A bit of notation is in order;  given vector fields $A$ and $B$, the vector denoted by $\left. d \exp(A)(B) \right|_{y}$ is meant to be the vector at the point $y$ obtained by transporting $B$ via the exponential mapping $\exp(A)$ (so in particular, it is the vector $B$ at $\exp(-A)(y)$ transported to $y$).

To establish the necessary properties of the lifting of $\Phi_{x_0}(t)$, two facts from geometry are required.  The first is the following:  suppose that $A$ is a vector field on $U$ which depends smoothly on some parameter $h$.  For any $x_0 \in U$ and any $s$ sufficiently small, the tangent vector of the curve $\gamma(h) := \exp( s A(h))(x_0)$ (as a function of $h$ for $s$ and $x_0$ fixed) is given by
\begin{equation} 
\frac{\partial}{\partial h} \gamma(h) = \left. \left( \int_0^1  d \exp((1 - \theta) s A(h)) \left( \frac{\partial A}{\partial h} \right)  d \theta \right) \right|_{\gamma(h)}.  \label{geom1}
\end{equation}
Equivalently, the tangent vector to the curve $\gamma(h)$ is given by transporting the vector
\begin{equation}
\left. \left( \int_0^1  d \exp( - \theta s A(h)) \left( \frac{\partial A}{\partial h} \right)  d \theta \right) \right|_{x_0} \label{geom3}
\end{equation}
to the point $\gamma(h)$ via the exponential map $\exp(s A(h))$.  The second fact needed is that, for any vector fields $A$ and $B$ on $U$; if $s$ is sufficiently small then
\begin{equation}
\frac{\partial}{\partial s} \left( \left. d \exp( s A ) (B) \right|_{x_0} \right) = \left. d \exp( s A ) ([B,A]) \right|_{x_0} \label{geom2}
\end{equation}
Note that equation \eqref{geom2} is nothing more than a computation of the Lie derivative of $B$ with respect to $A$ and can be found in Warner \cite{warner1971}, for example.  It is only \eqref{geom1} which requires a bit more explanation.  To that end, 
$\Gamma(s) := \exp(- s A(h)) \circ \exp(s A(h+\Delta h))(x_0)$.
\begin{align*}
\frac{\partial}{\partial s} g(\Gamma(s)) & = - \left. (A(h)g) \right|_{\Gamma(s)} + \left. \left( d \exp(-s A(h)) (A(h+\Delta h)) g \right) \right|_{\Gamma(s)} \\
& = \left. \left( d \exp(-s A(h)) (A(h+\Delta h)-A(h)) g \right) \right|_{\Gamma(s)}.
\end{align*}
It therefore follows that
\[ \frac{g(\Gamma(s)) - g(\Gamma(0))}{\Delta h} = \int_0^1 \left. \left( d \exp(-s \theta A(h)) \left( \frac{A(h+\Delta h) - A(h)}{\Delta h} \right) g \right) \right|_{\Gamma(s \theta)} d \theta. \]
As $\Delta h \rightarrow 0$, note that $\Gamma(s) \rightarrow x_0$. For fixed $s,h$, let $g (x) = f ( \exp(s A(h))(x))$ and let $\Delta h \rightarrow 0$.  The result is that
\[ \frac{\partial}{\partial h} f ( \exp(s A(h))(x_0)) = \left. \left( \int_0^1  d \exp( - \theta s A(h)) \left( \frac{\partial A}{\partial h} \right)  d \theta \right) \right|_{x_0} \! \! f (\exp(s A(h))(x_0)) \]
which is precisely what it means for the curve $\gamma(h) = \exp(s A(h))(x_0)$ to have a tangent vector obtained by transporting the vector \eqref{geom3} via the map $\exp(s A(h))$.

Consider the mapping $\varphi : \tilde U \rightarrow U$ given by $\varphi(s) := \exp(s \cdot X)(x_0)$.  Taking a Taylor expansion of the integrand \eqref{geom1} with respect to $\theta$ (computing these derivatives via \eqref{geom2}) allows one to easily compute derivatives of $\varphi$ with respect to the parameters $s_w$:
\begin{align*}
\frac{\partial}{\partial s_w} \varphi(s) & = \left. \left( \int_0^1 d \exp( (1-\theta) s \cdot X) \left( X_w \right) d \theta \right) \right|_{\varphi(s)} \\
&  = \left. \left( \sum_{j=0}^\infty \frac{(-1)^{j} [ (s \cdot X)^j X_w]}{(j+1)!} \right) \right|_{\varphi(s)}
\end{align*}
where $[(s \cdot X)^j X_w]$ is the repeated commutator given by $[(s \cdot X)^0 X_w] := X_w$ and $[(s \cdot X)^{j+1} X_w] = [s \cdot X,[(s \cdot X)^j X_w]]$ for each $j \geq 0$.
Note that this sum is, in fact, a finite sum by virtue of the vanishing commutator condition.  Also note that the coefficients of the sum are precisely the Taylor coefficients of $\frac{e^{-x}-1}{-x}$.

Now let $\tilde{c}_w(s)$ be the coefficient of $X_w$ when the sum
\[ \sum_{j=0}^\infty \frac{(-1)^j B_j [(s \cdot X)^j X_1]}{j!} \]
is expanded into a linear combination of words (where the $B_j$'s are the Bernoulli numbers; note that these are chosen so that the coefficients in $j$ are equal to the Taylor coefficients of $\frac{-x}{e^{-x}-1}$).  It follows that
\[ \left( \sum_{w \in {\cal N}} \tilde{c}_w(s) \frac{\partial}{\partial s_w} \right) \varphi(s) = \left. X_1 \right|_{\varphi(s)}. \]
The vector field $\sum_{w \in {\cal N}} \tilde{c}_w (s) \frac{\partial}{\partial s_w}$ is thus a lifting of $X_1$; moreover, since the coefficient $\tilde{c}_w(s)$ involves only those parameters $s_{w'}$ for which the length of $w'$ is less than the length of $w$, it follows that the integral curves of this lifted vector field in $\R^{\cal N}$ will be given by a polynomial function of the time parameter.  Thus composing these flows ($X_2$ may be lifted in precisely the same way) establishes the fact that $\Phi_{x_0}(t)$ has a lifting $\tilde{\Phi}_{x_0}(t)$ which is given coordinate-wise by polynomial functions.

To parametrize the fibers of $\varphi(s)$, the alternative formulation
\[ \frac{\partial}{\partial s_w} f (\varphi(s)) = \left. \left( \sum_{j=0}^\infty \frac{[(s \cdot X)^j X_w]}{(j+1)!} \right) \right|_{x_0} f ( \exp(s \cdot X)(x_0)) \]
is used (here the Taylor series expansion of the integrand of \eqref{geom3} is taken instead of \eqref{geom1}).  Suppose that constants $u_w$ are chosen so that $\sum_{w \in {\cal N}} u_w X_w$ equals the zero vector {\it at the particular point $x_0$} (the sum may or may not equal zero elsewhere).  If one defines coefficients $\tilde{d}_w(s,u)$ as before by expanding
\[ \sum_{w \in {\cal N}} \tilde{d}_w(s,u) X_w := \sum_{j=0}^\infty \frac{B_j [(s \cdot X)^j (u \cdot X) ]}{j!} \]
formally, it follows that the vector field $\sum_{w \in {\cal N}} \tilde{d}_w(s,u) \frac{\partial}{\partial s_w}$ will satisfy the property that
\[\left( \sum_{w \in {\cal N}} \tilde{d}_w(s,u) \frac{\partial}{\partial s_w} \right) \varphi(s) = 0. \]
In other words, for any appropriately chosen values of $u$, the corresponding vector field will be tangent to the fibers of $\varphi(s)$.  The integral curves of these fibers will be given by polynomial functions of $u$ for the same reason that the coefficients $\tilde{d}_w$ only depend on $s_{w'}$ for $w'$ of shorter length.  Now a simple dimension-counting guarantees that the fibers can, in fact, be smoothly parametrized by the flows of these vector fields, that is, the exponential flow of these vector fields based at $x_0$ has surjective differential with respect to the $u$ variables (one needs only apply the implicit function theorem; note that the curvature condition on the vector fields $X_1$ and $X_2$ guarantees that the differential $d \varphi$ is surjective).  Thus the earlier counting arguments hold, and in particular, $\Phi_{x_0}(t)=x$ has boundedly many solutions for all $x$ outside a set of measure zero.

To complete the proof of theorem \ref{averageop}, one fact remains to be established: namely, that the Jacobian determinant $J_{x_0}(t)$ is, up to a nonvanishing factor, a polynomial function of the parameters $t$.  Notice that the vector $\frac{\partial}{\partial t_i} \Phi_{x_0}(t)$ is equal to $d \exp(t_1 X_1) \circ \cdots \circ d \exp(t_{i-1} X_{i-1}) (X_i)$ evaluated at the point $\Phi_{x_0}(t)$.  Up to a bounded, nonvanishing factor, the Jacobian determinant can be evaluated by transporting these vectors to the point $x_0$ and then computing a determinant.  In that case, it follows that $J_{x_0}(t)$ is proportional to
\begin{align*}
 \det ( d & \exp(-t_{d+1} X_{d+1}) \circ \cdots \circ d \exp(-t_2 X_2)(X_1),  \\
& d \exp(-t_{d+1} X_{d+1}) \circ \cdots \circ d \exp(-t_3 X_3)(X_2), \\
& \left. \ldots, X_{d+1}) \right|_{x_0}. 
\end{align*}
But now \eqref{geom2} guarantees that each vector in this expression is a polynomial function of $t$, and multilinearity of the determinant establishes the desired property of $J_{x_0}(t)$.

\bibliography{mybib}

\end{document}